\documentclass{article} 
\usepackage[utf8]{inputenc}
\usepackage[T1]{fontenc}
\usepackage{lmodern}
\usepackage[a4paper]{geometry}
\usepackage{babel}
\usepackage{enumitem}
\usepackage{fullpage}
\usepackage{pifont}
\usepackage[usenames, dvipsnames]{xcolor}
\usepackage{tikz}
\usetikzlibrary{patterns}
\usepackage[justification=justified]{caption}
\usepackage{pgfplots}
\pgfplotsset{compat=1.15}
\usepackage{mathrsfs}
\usetikzlibrary{arrows}
\usepackage{amssymb,amsmath,mathtools,amsthm,dsfont}
\usepackage{stmaryrd}

\usepackage[all]{xy}

\newtheorem{theoreme}{Théorème}[section]
\newtheorem{theorem}[theoreme]{Theorem}
\newtheorem{prop-f}[theoreme]{Proposition}

\newtheorem{lemma}[theoreme]{Lemma}

\newtheorem{definition}[theoreme]{Definition}
\newtheorem{rem}[theoreme]{Remarque}

\newcommand{\E}{\mathbb{E}}

\newcommand{\N}{\mathbb{N}}

\renewcommand{\P}{\mathbb{P}}

\newcommand{\R}{\mathbb{R}}

\newcommand{\Z}{\mathbb{Z}}

\renewcommand{\L}{\mathcal{L}}

\newcommand{\cC}{\mathcal{C}}

\newcommand{\cF}{\mathcal{F}}

\newcommand{\ent}{\lfloor nt \rfloor}
\newcommand{\ens}{\lfloor ns \rfloor}
\newcommand{\ents}{\lceil nt \rceil}

\newcommand{\dist}{\mathrm{dist}}
\newcommand{\sqn}{\sqrt{n}}
\newcommand{\bt}{_{|t}}
\newcommand{\btn}{_{|t_n}}
\newcommand{\fS}{\mathfrak{S}}
\newcommand{\oS}{\overline{S}}

\renewcommand{\d}{\mathrm{d}}
\newcommand{\Lip}{\mathrm{Lip}}

\def\indic{\mathbf 1}

\renewcommand{\epsilon}{\varepsilon}
\renewcommand{\phi}{\varphi}

\numberwithin{equation}{section}
\makeindex
\begin{document}

\title{Rate of convergence of the conditioned random walk towards the Brownian bridge}
\author{L. Decreusefond\footnote{Telecom Paris, laurent.decreusefond@telecom-paris.fr} \and A. Jacquet\footnote{Telecom Paris, antonin.jacquet@telecom-paris.fr}}
\date{}
\maketitle
\begin{abstract}
	We study the rate of convergence of two discrete processes towards the Brownian bridge: the random walk conditioned to be zero at time $2n$ and the empirical process which appears in the Glivencko-Cantelli theorem. Combining a functional Stein method with  a Radon–Nikodym representation of the bridge, we bound the Fortet–Mourier distance between these conditioned processes and the Brownian bridge.

	% For Rademacher and “Poisson minus one” increments we obtain the rate $O(n^{-1/18} \log n)$. More generally, for lattice distributions with unit variance we derive a bound of order $\max(n^{-1/18} \log n, \rho(L,n), \tau(L,n,\delta))$, where $\rho$ and $\tau$ are explicit quantities measuring local and Fourier approximation errors. The proofs rest on characteristic-function estimates and careful control of the Radon–Nikodym derivatives.
\end{abstract}

\section{Introduction}\label{s: section introduction.}

The Stein's method is a powerful tool to derive rates of convergence in distributional limit theorems. It is mainly applicable to distances between probability measures that can be expressed as
\begin{equation*}
	\text{dist}_{\mathcal{F}}(\mu,\nu)=\sup_{F\in \mathcal{F}} \left( \int F \d\mu-\int F \d\nu \right),
\end{equation*}
where $\mathcal{F}$ is a suitable class of test functions. When $\mathcal{F}$
is the set of Lipschitz functions with Lipschitz constant bounded by $1$, the corresponding distance is the Wasserstein-1 or Kantorovitch-Rubinstein distance \cite{Villani2003}. Choosing $\mathcal{F}$ as the class $\{\mathbf 1_{(-\infty,x]},x\in \R\}$, $\dist_{\mathcal{F}}$ yields the Kolmogorov distance. The Fortet-Mourier distance, which defines the topology of convergence in law is obtained when $\mathcal{F}$ is the set of bounded and Lipschitz  functions, see~\cite{Dudley2002}. The first step of the deployment of Stein's method consists in identifying a functional characterization of the target distribution~$\mu$. This is usually done by finding a Markov process for which $\mu$ is the stationary distribution, and then computing its generator. This task has been achieved for a limited number of target measures, including the normal distribution \cite{Decreusefond2015}, the Poisson distribution,  stable distributions \cite{Chen_2021} and more recently max-stable distributions \cite{Costaceque24_cpn,Costaceque24,Costaceque24_mlti}.

When the target distribution corresponds to a stochastic process, such as Brownian motion, the situation is more delicate. The generator has been constructed in \cite{Coutin2013} when we consider that the Brownian motion is an element of fractional Sobolev spaces. In \cite{Shih2011}, the theory has been established when the Brownian motion is viewed as an element of the space of continuous functions on $[0,T]$.
However, both constructions require to compute a trace term which is notoriously difficult to estimate.
In \cite{CouDec2020}, the authors circumvented this issue by discretizing all the processes under consideration. The distance between the processes and their affine interpolation is estimated using sample-path inequalities, whereas the distance between the affine interpolations is bounded using the multivariate Stein's method. The global conclusion of these previous developments is that it is highly challenging to apply Stein's method directly to infinite-dimensional processes. In \cite{besancon:hal-03283778}, several functional limit theorems for Markov processes are obtained by using transformations of Poisson measures and Brownian motions. Formally speaking, to estimate the distance between $\mu$ and $\nu$, one first finds a transformation $T$ and two probability measures $\mu'$ and $\nu'$ such that $T\mu'=\mu$ and $T\nu'=\nu$. If $T$ is Lipschitz continuous (which is rather common  in the usual applications), then
\begin{equation*}
	\text{dist}_{\Lip_1}(\mu,\nu)= \|T\|_\Lip\ \text{dist}_{\Lip_1}(\mu',\nu'),
\end{equation*}
where $\Lip_1$ is the set of Lipschitz functions with Lipschitz constant bounded by $1$. It then remains to estimate the distance between $\mu'$ and $\nu'$  using Stein's method. It is proved in particular that the random walk forced to be $0$ at time $2n$ converges to the Brownian bridge with rate $O(n^{-1/6} \log n)$ in the Wasserstein-1 distance on $\cC([0,1],\R)$ since the map
\begin{align*}
	T: \cC([0,1],\R) & \to \cC([0,1],\R)                              \\
	f                & \mapsto \left( t \mapsto f(t) - t f(1) \right)
\end{align*}
is $1$-Lipschitz from $\cC([0,1],\R)$ to itself. But this does not answer the question of determining the rate of convergence of the law of the random walk conditioned to be $0$ at time $2n$ towards the law of the Brownian bridge.

This is the motivation of our current work. The main difficulty is to represent the conditional laws in a way that make the computations tractable. It turns out that the crucial remark is to note that for any time $t<1$, the law of the Brownian bridge on $[0,t]$ is absolutely continuous with respect to the law of the ordinary Brownian motion up to time $t$. Furthermore, the Radon-Nikodym derivative has a known expression which reveals to be a Lipschitz continuous function of the Brownian sample-path. On the remaining interval $[t,1]$, the Brownian bridge cannot be large by continuity and the fact that it is zero at time $1$. The same considerations apply to the random walk.
Then, by choosing conveniently the time $t$ as a function of $n$ and using \cite{CouDec2020}, we can find the rate of convergence of the conditioned random walk towards the Brownian bridge. For technical reasons, we need to restrict our analysis the Fortet-Mourier distance, which is slightly stronger than the Wasserstein-1 distance. The same approach allows us to derive a rate of convergence for the empirical process of the Glivenko-Cantelli theorem towards the Brownian bridge since the empirical process can be represented as a Poisson random walk conditioned to be $0$ at time $1$. The actual result holds for more general lattice distributions and the rate of convergence depends on the convergence rates of some local limits of the increments distribution.
It is clear that the absolute continuity of the Brownian bridge with respect to the Brownian motion can not hold up to time $1$ and the price to pay to handle the singularity at time $1$ is a loss in the rate of convergence.
The paper is organized as follows. Section \ref{s: section 2.} presents the main results, Theorem \ref{thm: theoreme principal rademacher.} and Theorem \ref{thm: theoreme principal GC.}, with Theorem \ref{thm: theoreme principal GC.} being an immediate consequence of Theorem \ref{thm: theoreme principal poisson moins 1.}.
These theorems follow from the result for more general lattice distributions, which is stated in Section \ref{s:section3} and requires the introduction of several auxiliary objects. All proofs are given in Section \ref{s: section 4.}.

\section{Main results}\label{s: section 2.}
The estimates established in this paper are formulated in terms of the Fortet–Mourier distance. Recall that this distance is defined by
\begin{equation*}
	\dist_{F.M.}(\mu,\nu) = \sup_{F \in \cF} \left( \int F \mathrm{d}\mu - \int F \mathrm{d}\nu  \right),\label{eq: definition de la distance de Fortet-Mourier.}
\end{equation*}
where
\begin{multline*}
	\cF = \left\{F \, : \, \cC([0,1],\R) \to \R, \, \|F\|_\infty \le 1 \text{ and } |F(\omega_1) - F(\omega_2)| \le \|\omega_1-\omega_2\|_\infty, \, \right. \\ \left. \forall \omega_1,\omega_2 \in \cC([0,1],\R)\right\}.
\end{multline*}
Let $(X_i, i \ge 1)$ be a sequence of independent, centered, square integrable, and identically distributed $\R$-valued random variables, whose common distribution is denoted by $\L$. Define for every $n \ge 1$, $S_n$ as the continuous process on $[0,1]$ obtained by interpolating the process $\oS_n$ defined for every $t \in [0,1]$ by
\begin{equation}
	\oS_n(t) = \frac{1}{\sqrt{n}} \sum_{i=1}^{\ent} X_i\label{eq: processus utilise pour la definition de Sn.}
\end{equation}
between the points $\displaystyle \left\{ \frac{k}{n}, \, k \in \{0,\dots,n\} \right\}$.
For any $x \in \R$, denote by $\{x\}$ the fractional part of $x$.
Then, for every $t \in [0,1]$ and $n \ge 1$ by
\begin{equation}
	S_n(t) = \frac{1}{\sqrt{n}} \sum_{i=1}^{\ent} X_i + \frac{\{nt\}}{\sqrt{n}}X_{\ent+1}.\label{eq: definition de Sn.}
\end{equation}
The first main result is the following theorem, which treats the case where $\L$ is the Rademacher distribution. Its proof is given in Section \ref{s: sous-section preuve Rademacher.}.
\begin{theorem}\label{thm: theoreme principal rademacher.}
	Assume that $\L$ is the Rademacher distribution, i.e.
	\begin{equation*}
		\P(X_1=1) = \P(X_1=-1) = \frac{1}{2}.
	\end{equation*}
	Then, there exists a constant $C>0$ such that for all $n \ge 1$,
	\begin{equation*}
		\dist_{F.M.}\left(\mu_n(\L),B^{br}\right)\le C n^{-\frac{1}{18}} \log(n),
	\end{equation*}
	where $\mu_n(\L)$ is the distribution of $S_{2n}$ conditioned to be $0$ at time $1$, $S_{2n}$ is the process defined at \eqref{eq: definition de Sn.} and $B^{br}$ is the Brownian bridge.
\end{theorem}
Another process converging to the Brownian bridge is the empirical process.
Let $(U_i, \, i \ge 1)$ be a sequence of i.i.d.\ random variables uniformly distributed on $[0,1]$ and let $F$ be their common distribution function.
For every $n \ge 1$, define the cumulative empirical function $F_n$ for every $t \in [0,1]$ by:
\begin{equation*}
	F_n(t) = \frac{1}{n}\sum_{i=1}^n \indic_{\{U_i \le t\}}.\label{eq: definition des fonctions cumulatives.}
\end{equation*}
Define $b_n$ the continuous empirical process on $[0,1]$ obtained by interpolating the empirical process $\overline{b}_n$ defined for every $t \in [0,1]$ by:
\begin{equation}
	\overline{b}_n(t) = \sqrt{n} \left( F_n(t) - F(t) \right)\label{eq: definition du processus empirique.}
\end{equation}
between the points $\displaystyle \left\{\frac{k}{n}, \, k \in \{0,\dots,n\}\right\}$.
By the strong law of large numbers, the sequence of processes $(F_n)$ converges almost surely to $F$ at any single point. Then, the Glivencko--Cantelli theorem gives a uniform law of large numbers:
\begin{equation*}
	\sup_{t \in [0,1]} \left| F_n(t) - F(t) \right| \xrightarrow[n \to \infty]{} 0 \text{ a.s.}
\end{equation*}
Furthermore, Donsker's theorem (see \cite{DonskerGlivenkoCantelli}) asserts that the empirical process $\overline{b}_n$ defined above converges in distribution to the Brownian bridge on the space of c\`ad l\`ag functions on $[0,1]$ equipped with the Skorokhod topology.
The following theorem, which is the second main result of this paper, provides a rate of convergence.
\begin{theorem}\label{thm: theoreme principal GC.}
	Let $(U_i, \, i \ge 1)$ be a sequence of i.i.d. random variables uniformly distributed on $[0,1]$ and $b_n$ the continuous empirical process defined at \eqref{eq: definition du processus empirique.}.
	Then, there exists a constant $C>0$ such that for all $n \ge 1$,
	\begin{equation*}
		\dist_{F.M.}\left(\tilde{\mu}_n(\L),B^{br}\right)\le C n^{-\frac{1}{18}} \log(n),
	\end{equation*}
	where $\tilde{\mu}_n(\L)$ is the distribution $b_n$ and $B^{br}$ is the Brownian bridge.
\end{theorem}
Theorem \ref{thm: theoreme principal GC.} is in fact an immediate corollary of Theorem \ref{thm: theoreme principal poisson moins 1.} just below.
As it is stated in \cite{DonskerGlivenkoCantelli} and in \cite{Marckert}, a simple reasoning proves that
\begin{equation*}
	\left( b_n(t) \right)_{t \in [0,1]} \stackrel{\text{d}}{=} \left( S_n(t) \right)_{t \in [0,1]} \text{ conditioned by $S_n=0$,}
\end{equation*}
where $S_n$ is the continuous centered Poisson random walk. This means that $S_n$ is the continuous process defined for every $t \in [0,1]$ as at \eqref{eq: definition de Sn.}, where for all $i \ge 1$, $X_i + 1$ is a random variable with a Poisson distribution of parameter $1$.
Thus, Theorem \ref{thm: theoreme principal GC.} is proved as soon as Theorem \ref{thm: theoreme principal poisson moins 1.} is established. The proof of Theorem \ref{thm: theoreme principal poisson moins 1.} is given in Section \ref{s: sous-section preuve poisson moins un.}.
\begin{theorem}\label{thm: theoreme principal poisson moins 1.}
	Assume that $\L$ is the Poisson minus one distribution of parameter $1$, that is, for all $i \ge 1$, $P_i := X_i + 1$ is a random variable with a Poisson distribution of parameter $1$.
	Then, there exists a constant $C > 0$ such that for all $n \ge 1$,
	\begin{equation*}
		\dist_{F.M.}\left(\mu_n(\L),B^{br}\right)\le C n^{-\frac{1}{18}} \log(n),
	\end{equation*}
	where $\mu_n(\L)$ is the distribution of $S_{n}$ conditioned to be $0$ at time $1$ and $B^{br}$ is the Brownian bridge.
\end{theorem}

\section{General strategy}\label{s:section3}
%We state a more general result, valid for lattice distributions. This is Theorem \ref{thm: theorem principal general.}, from which Theorem \ref{thm: theoreme principal rademacher.} and Theorem \ref{thm: theoreme principal poisson moins 1.} follow. Before stating it, we introduce some definitions and notations.

\subsection{Definitions and notations}\label{s: sous-section notations.}

\subsubsection{General notations}

For a function $f \in \cC([0,1],\R)$, denote by $f\bt$ the function of $\cC([0,1],\R)$ defined for every $s \in [0,1]$ by
\begin{equation}\label{eq: notation f barre t.}
	f\bt(s) = \left\{
	\begin{array}{ll}
		f(s) & \text{ if } s<t,     \\
		f(t) & \text{ if } s \ge t.
	\end{array}
	\right.
\end{equation}
For $E \subset \R$, the set $-E$ is defined by
\begin{equation*}
	-E = \{ -x, \, x \in E \}
\end{equation*}
For a sequence $(X_i, i \ge 1)$ of independent random variables with some common distribution $\L$, the processes $S_n$ and $\oS_n$ are already defined at \eqref{eq: definition de Sn.} and \eqref{eq: processus utilise pour la definition de Sn.}. We also define the process $\oS^+_n$ for every $t \in [0,1]$ by
\begin{equation}
	\oS^+_n(t) = \frac{1}{\sqrt{n}} \sum_{i=1}^{\ents} X_i\label{eq: processus plus utilise pour la definition de Sn.}
\end{equation}

\subsubsection{Definitions and notations related to the distribution $\L$. }\label{s: sous-section definitions et notations pour la loi L.}

The distribution $\L$ of the sequence $(X_i, i \ge 1)$ is a lattice distribution if there exist two constants $b$ and $h>0$ such that
\begin{equation}
	\P (X_1 \in b+h\Z) = 1, \label{eq: definition lattice distribution.}
\end{equation}
where $b + h\Z = \{b+hz, \, z \in \Z\}$.
The largest $h$ for which \eqref{eq: definition lattice distribution.} holds is the span of the distribution $\L$.
\begin{definition}
	We define $\kappa(\L)$ as the smallest $k \in \N^*$ such that for all $n \ge 1$,
	\begin{equation*}
		\P \left( S_{kn}(1) = 0 \right) > 0.
	\end{equation*}
\end{definition}
For example,
\begin{itemize}
	\item if $\L$ is the Rademacher distribution, $\kappa(\L) = 2$,
	\item if $X_1 = P_1 - 1$ where $P_1$ is a Poisson random variable with parameter $1$, then $\kappa(\L) = 1$.
\end{itemize}
When there is no ambiguity about the distribution in question, we write $\kappa$ instead of $\kappa(\L)$.
We denote by $\mu_n(\L)$ the distribution of $S_{\kappa(\L)n}$ conditioned to be $0$ at time $1$.
For every $t \in [0,1]$ and $x \in \R$, we denote
\begin{equation}
	p^n_t(x) = \P \left( \oS_n(t) = x \right).\label{eq: definition pnt.}
\end{equation}
For every $t \in [0,1]$, we denote by $p_t$ the probability density function of the Brownian motion at time $t$. Thus, for every $x \in \R$,
\begin{equation}
	p_t(x) = \frac{1}{\sqrt{2 \pi t}}\exp \left( - \frac{x^2}{2t} \right).\label{eq: definition pt.}
\end{equation}
For $\L$ a lattice distribution with span $h$, for $n \ge 1$ and $t \in [0,1]$, denote by $\fS_n(t)$ the set
\begin{equation*}
	\fS_n(t) = \left\{ \frac{\ent b + hz}{\sqrt{n}}, \, z \in \Z \right\},
\end{equation*}
where $b$ is such that \eqref{eq: definition lattice distribution.} holds.
Note that for all $n \ge 1$ and $t \in [0,1]$,
\begin{equation*}
	\P \left(\oS_n(t) \in \fS_n(t) \right) = 1.
\end{equation*}
\begin{definition}
	For a lattice distribution $\L$ with span $h$, define $\rho(\L,n)$ by
	\begin{equation*}
		\rho(\L,n) = n^{\frac{1}{18}} \left| \frac{\sqrt{\kappa(\L)n}}{h} p^{\kappa(\L)n}_1(0) - p_1(0) \right|.
	\end{equation*}
\end{definition}
\begin{definition}
	For a distribution $\L$ and for $\displaystyle \delta > \frac{1}{18}$, define $\tau(\L,n,\delta)$ by
	\begin{equation*}
		\tau(\L,n,\delta) = \int_{-n^\delta}^{n^\delta} \left| \phi^{\ens} \left( \frac{u}{\sqrt{n}} \right) - \exp \left( - \frac{s u^2}{2} \right) \right| \mathrm{d}u.
	\end{equation*}
\end{definition}

\subsection{General result}
The third main result of this paper is the following, valid for more general lattice distributions. Its proof is in Section \ref{s: sous-section preuve general.}.
\begin{theorem}\label{thm: theorem principal general.}
	Assume that $\E[X]=0$ and $\displaystyle \E \left[X^2\right] =1$, where $X$ denotes a random variable with distribution $\L$. Furthermore, assume that $\L$ is a lattice distribution.
	Then, for every $\displaystyle \delta > \frac{1}{18}$, there exists $C>0$ such that for all $n \ge 1$,
	\begin{equation*}
		\dist_{F.M.}\left(\mu_n(\L),B^{br}\right)\le C \max\left( n^{-\frac{1}{18}}\log(n), \rho(\L,n), \tau(\L,n,\delta)\right),
	\end{equation*}
	where $\mu_n(\L)$ is the distribution of $S_{\kappa(\L)n}$ conditioned to be $0$ at time $1$ and $B^{br}$ is the Brownian bridge.
\end{theorem}

\begin{rem}
	The above theorem is meaningful only if one can estimate the rate of convergence of $\rho(\mathcal{L},n)$ towards $0$ as $n \to \infty$ and that of $\tau(\mathcal{L},n,\delta)$ for some $\displaystyle \delta > \frac{1}{18}$.
\end{rem}

\section{Proofs}\label{s: section 4.}

In the proofs, we write $a \lesssim b$ if there exists a constant $C>0$ such that $a \le C b$.
By default, the constant $C$ depends only on the distribution $\L$.
Whenever necessary, we explicitly indicate at the start of a proof the quantities on which $C$ may depend.

\subsection{Proof of Theorem \ref{thm: theorem principal general.}}\label{s: sous-section preuve general.}

For the remainder of this section, we fix a lattice distribution $\L$ with span $h$. Recall that $\kappa(\L)$ is defined in Section \ref{s: sous-section definitions et notations pour la loi L.}. In the following, we write $\kappa$ instead of $\kappa(\L)$.
Recall that the notation $f\bt$ for every function $f \in \cC([0,1],\R)$ is defined in Section \ref{s: sous-section notations.} at \eqref{eq: notation f barre t.}.
As mentioned in Section~\ref{s: section introduction.}, the crucial tool in the proof is the following equality, which follows from the fact that, for any $t<1$, the law of the Brownian bridge on $[0,t]$ is absolutely continuous with respect to the law of standard Brownian motion up to time $t$. For every $F \in \cF$ and for every $t \in (0,1)$.
\begin{equation}
	\E \left[ F \left( B^{br}\bt \right) \right] = \E \left[ F \left( B\bt \right) \frac{p_{1-t}(-B_t)}{p_1(0)} \right],\label{eq: egalite esperance du pont brownien.}
\end{equation}
where $p_s$ is defined for every $s \in [0,1]$ at \eqref{eq: definition pt.}.
A proof of \eqref{eq: egalite esperance du pont brownien.} can be found for instance in Section VIII.3 in \cite{bertoin1996levy}.
An analogous of this equality is naturally expected for the conditioned random walk. This is given in Lemma \ref{l: lemme egalite esperance de la marche aleatoire conditionnee.} below.

\begin{lemma}\label{l: lemme egalite esperance de la marche aleatoire conditionnee.}
	Assume that $\L$ satisfies the assumptions of Theorem \ref{thm: theorem principal general.}.
	Then, for every $F \in \cF$, $t \in (0,1)$ and $n \ge 1$,
	\begin{equation}
		\E \left[ F \left( {S_{\kappa n}}\bt \right) | S_{\kappa n}(1)=0 \right] = \E \left[  F \left( {S_{\kappa n}}\bt \right) \frac{p^{\kappa n}_{1-t}(-\oS^+_{\kappa n}(t))}{p^{\kappa n}_1(0)} \right],\label{eq: egalite esperance de la marche aleatoire conditionnee.}
	\end{equation}
	where for all $n \ge 1$, $S_n$ is defined at \eqref{eq: definition de Sn.}, $\oS^+_n$ at \eqref{eq: processus plus utilise pour la definition de Sn.} and for all $s \in [0,1]$, $p^n_s$ at \eqref{eq: definition pnt.}.
\end{lemma}

\begin{proof}
	%The proof of this lemma relies solely on basic definitions and on the fact that we are working with discrete random variables.
	%The few complications arise from the fact that we must be careful because $\kappa nt$ is not necessarily an integer.
	%The reader can follow this proof without thinking about these complications. 
	Let $F \in \cF$, $t \in (0,1)$ and $n \ge 1$.
	Under the assumptions of Lemma \ref{l: lemme egalite esperance de la marche aleatoire conditionnee.}, $\L$ is a lattice distribution. Thus, the set of values in $\cC([0,1],\R)$ that can be taken by ${S_{\kappa n}}$ is finite. Denote it by $\Pi_n$. To prove \eqref{eq: egalite esperance de la marche aleatoire conditionnee.}, it suffices to show that for every $\pi \in \Pi_n$,
	\begin{equation}
		\P \left( {S_{\kappa n}}\bt = \pi\bt | S_{\kappa n}(1)=0 \right) = \P\left({S_{\kappa n}}\bt = \pi\bt\right) \frac{p^{\kappa n}_{1-t}(-\pi([t]))}{p^{\kappa n}_1(0)},\label{eq: equation bypass conditionning 1.}
	\end{equation}
	where $\displaystyle [t] = \frac{\lceil \kappa nt \rceil}{\kappa n}$.
	Now, proving \eqref{eq: equation bypass conditionning 1.} is the same as proving the following:
	\begin{equation*}
		\P \left( \oS_{\kappa n}(1-t) = -\pi([t]) \right) =: p^{\kappa n}_{1-t}(-\pi([t])) = \P \left( S_{\kappa n}(1) = 0 | {S_{\kappa n}}\bt = \pi\bt \right).
	\end{equation*}
	However, this equality is almost obvious. Indeed, using the definition of $S_{\kappa n}$ and the fact that the knowledge of ${S_{\kappa n}}\bt$ is equivalent to the knowledge of $X_i$ for $i \in \{1,\dots,\lceil \kappa nt \rceil \}$,
	\begin{align*}
		\P \left( S_{\kappa n}(1) = 0 | {S_{\kappa n}}\bt = \pi\bt \right) & = \P \left( \sum_{i = \lceil \kappa nt \rceil + 1}^{\kappa n} X_i = - \pi([t]) \right)                                                    \\
		                                                                   & = \P \left( \sum_{i = 1}^{\lfloor \kappa n(1-t) \rfloor} X_i = - \pi([t]) \right) \text{ since the family $(X_i, \, i \ge 1)$ is i.i.d.,} \\
		                                                                   & =\P \left( \oS_{\kappa n}(1-t) = -\pi([t]) \right),
	\end{align*}
	which concludes the proof.
\end{proof}
Before proceeding to the proof of Theorem~\ref{thm: theorem principal general.}, we state several auxiliary lemmas that will be used throughout the proof.
\begin{lemma}\label{c: claim pour majorer Aunntn.}
	For every $n \ge 1$,
	\begin{equation}
		\E \left[ \|S_{\kappa n} - {S_{\kappa n}}_{|t_n} \|_\infty | S_{\kappa n}(1)=0 \right] \le 2 \E \left[ \|{S_{\kappa n}}_{|1-t_n} \|_\infty | S_{\kappa n}(1)=0 \right],\label{eq: claim pour majorer Aunntn.}
	\end{equation}
	where $\displaystyle t_n = 1 - n^{-\frac{1}{9}}$.
\end{lemma}

\begin{lemma}\label{l: lemme convergence des fonctions de probas.}
	Assume that $\E[X]=0$ and $\displaystyle \E \left[X^2\right] =1$, where $X$ denotes a random variable with distribution $\L$. Furthermore, assume that $\L$ is a lattice distribution with span $h$. Then, for every $\displaystyle \delta > \frac{1}{18}$ and for every $\displaystyle \eta < \min(2\delta,1)$, there exists $C>0$ such that for all $n \ge 1$, for all $\displaystyle s \in \left[ n^{-\eta},1 \right]$,
	\begin{equation*}
		\sup_{a \in \fS_n(s)} \left|\frac{\sqrt{n}}{h} p^n_s(a) - p_s(a)\right| \le C \max \left( n^{-\frac{1}{18}}, \tau(\L,n,\delta)\right).
	\end{equation*}
\end{lemma}
%\paragraph*{Functions $L_t$ and $L^n_t$.}
Recall the definitions given in Section \ref{s: sous-section notations.}.
For every $t \in [0,1]$, define the function $L_t \, : \, \R \to \R$ for all $x \in \R$ by
\begin{equation}
	L_t(x) = \frac{p_{1-t}(-x)}{p_1(0)}.\label{eq: definition de Lt.}
\end{equation}
For every $t \in [0,1]$ and $n \in \N^*$, define similarly the function $L^n_t \, : \, \R \to \R$ for all $x \in \R$ by
\begin{equation}
	L^n_t(x) = \frac{p^{\kappa n}_{1-t}(-x)}{p^{\kappa n}_1(0)}.\label{eq: definition de Lnt.}
\end{equation}

\begin{lemma}\label{l: lemme implication des hypotheses.}
	Assume that $\L$ satisfies the assumptions of Theorem \ref{thm: theorem principal general.}. Then, for every $\delta > \frac{1}{18}$, there exists $C>0$ such that for all $n \ge 1$ and all $\displaystyle t \in \left[ 0, 1-n^{-\frac{1}{9}} \right]$,
	\begin{equation}
		\sup_{a \in -\fS_{\kappa n}(1-t)} |L_t(a)-L^n_t(a)| \le C \max \left( n^{-\frac{1}{18}}, \rho(\L,n), \tau(\L,n,\delta)\right).\label{eq: lemme implication des hypotheses.}
	\end{equation}
	Thus, for every $\delta > \frac{1}{18}$, there exists $C>0$ such that for all $n \ge 1$ and all $\displaystyle t \in \left[ 0, 1-n^{-\frac{1}{9}} \right]$,
	\begin{equation}
		\E \left[ \left| L_t\left(\oS^+_{\kappa n}(t)\right) - L^n_t\left(\oS^+_{\kappa n}(t)\right) \right| \right] \le C \max \left( n^{-\frac{1}{18}}, \rho(\L,n), \tau(\L,n,\delta)\right).\label{eq: hypothese esperance de la difference des fonctions L et Ln.}
	\end{equation}
\end{lemma}
We now turn to the proof of Theorem \ref{thm: theorem principal general.}. The proofs of Lemma \ref{c: claim pour majorer Aunntn.}, Lemma \ref{l: lemme convergence des fonctions de probas.} and Lemma \ref{l: lemme implication des hypotheses.} are provided immediately afterwards.
\begin{proof}[Proof of Theorem \ref{thm: theorem principal general.}]
	Assume that $\L$ satisfies the assumptions of Theorem \ref{thm: theorem principal general.}.
	Fix $\displaystyle \delta > \frac{1}{18}$.
	In this proof, we write $a \lesssim b$ if there exists a constant $C>0$ depending only on $\L$ and $\delta$ such that $a \le C b$.
	We begin by the following. For every $\displaystyle t \in \left(0,1\right)$,
	\begin{align*}
		\dist_{F.M.} \left( \mu_n(\L),B^{br} \right) & = \sup_{F \in \cF} \left( \E \left[ F \left( S_{\kappa n} \right) | S_{\kappa n} (1) = 0 \right] - \E \left[ F \left( B^{br} \right) \right] \right) \\
		                                             & \le \sup_{F \in \cF} A_1(F,n,t) + \sup_{F \in \cF} A_2(F,n,t) + \sup_{F \in \cF} A_3(F,t),
	\end{align*}
	where
	\begin{align*}
		A_1(F,n,t) & = \E \left[ F \left( S_{\kappa n} \right) | S_{\kappa n} (1) = 0 \right] - \E \left[ F \left( {S_{\kappa n}}\bt \right) | S_{\kappa n} (1) = 0 \right], \\
		A_2(F,n,t) & = \E \left[ F \left( {S_{\kappa n}}\bt \right) | S_{\kappa n} (1) = 0 \right] - \E \left[ F \left( B^{br}\bt \right) \right],                           \\
		A_3(F,t)   & = \E \left[ F \left( B^{br}\bt \right) \right] - \E \left[ F \left( B^{br} \right) \right].
	\end{align*}
	We complete the proof by bounding from above each term $A_1(F,n,t)$, $A_2(F,n,t)$ and $A_3(F,t)$ by $\displaystyle C \max\left( n^{-\frac{1}{18}}\log(n), \rho(\L,n), \tau(\L,n,\delta)\right)$ for every $n \ge 1$ and for some constant $C>0$ independent of $F$, $t$ and $n$.
	The variable $t$, in turn, depends on $n$, and in order to obtain the most optimal upper bound, we set, for each $n \ge 1$,
	\begin{equation}
		t_n=1-n^{-\frac{1}{9}}.\label{eq: preuve generale definition tn.}
	\end{equation}
	We henceforth fix $F \in \cF$ for the rest of the proof. With $F$ fixed, we simplify the notation by writing $A_1(n,t)$, $A_2(n,t)$, and $A_3(t)$ in place of $A_1(F,n,t)$, $A_2(F,n,t)$, and $A_3(F,t)$.
	We begin by the bound for $A_3(t_n)$.
	%\paragraph*{Bound for $A_3(t_n)$.}
	Since $F$ is $1$-Lipschitz, we have
	\begin{equation}
		A_3(t_n) \le \E \left[ \left\|B_{|t_n}^{br} - B^{br}\right\|_\infty \right] = \E \left[ \sup_{s \in [t_n,1]} \left| B^{br}_{t_n} - B^{br}_s \right| \right].\label{eq: preuve generale majoration.}
	\end{equation}
	Recall that $t \mapsto B^{br}_t, \, t\in[0,1]$ has the same distribution as $t \mapsto B_t-tB_1, \, t\in[0,1]$. For a reference, one can for example see (8.4.3) in \cite{Durrett2019}.
	This, and \eqref{eq: preuve generale majoration.}, give
	\begin{equation}
		A_3(t_n) \le \E \left[ \sup_{s \in [t_n,1]} \left| B_{t_n} - B_s \right| \right] + (1-t_n) \E \left[ |B_1| \right].\label{eq: preuve generale majoration 2.}
	\end{equation}
	Then, using the classic properties of Brownian motion, namely the properties of stationary increments and of self-similarity, we get that $\displaystyle \sup_{s \in [t_n,1]} \left| B_{t_n} - B_s \right|$ and $\displaystyle \left| 1-t_n \right|^{\frac{1}{2}} \sup_{s \in [0,1]} |B_s|$ have the same distribution. Thus, since $\E \left[ \sup_{s \in [0,1]} \left| B_s \right| \right] < \infty$, for every $n \ge 1$, the following inequality holds:
	\begin{equation}
		\E \left[ \sup_{s \in [t_n,1]} \left| B_{t_n} - B_s \right| \right] \lesssim (1-t_n)^{\frac{1}{2}}.\label{eq: preuve generale majoration 3.}
	\end{equation}
	%One can refer to {\color{orange}} for a reference on the classic properties of the Brownian motion. 
	Finally, combining \eqref{eq: preuve generale majoration 3.} and \eqref{eq: preuve generale majoration 2.} gives that for all $n \ge 1$,
	\begin{equation*}
		A_3(t_n) \lesssim (1-t_n)^{\frac{1}{2}} \log(n).
	\end{equation*}
	With the definition of $t_n$ given at \eqref{eq: preuve generale definition tn.}, we obtain for all $n \ge 1$,
	\begin{equation*}
		A_3(t_n) \lesssim n^{-\frac{1}{18}} \log(n) \le \max\left( n^{-\frac{1}{18}}\log(n), \rho(\L,n), \tau(\L,n,\delta)\right).
	\end{equation*}
	% \paragraph*{Bound for $A_1(n,t_n)$.}
	The next step is the bound for $A_1(n,t_n)$.
	Since $F$ is $1$-Lipschitz, we have
	\begin{equation*}
		A_1(n,t_n) \le \E \left[ \|S_{\kappa n} - {S_{\kappa n}}_{|t_n} \|_\infty | S_{\kappa n}(1)=0 \right].
	\end{equation*}
	Applying Lemma \ref{c: claim pour majorer Aunntn.}, we get for all $n \ge 1$,
	\begin{equation}
		A_1(n,t_n) \le 2 \E \left[ \|{S_{\kappa n}}_{|1-t_n} \|_\infty | S_{\kappa n}(1)=0 \right].\label{eq: preuve generale majoration A1 1.}
	\end{equation}
	Now, to compare the right-hand side of the above inequality with $\E \left[ \|{S_{\kappa n}}_{|1-t_n} \|_\infty \right]$, we use Lemma \ref{l: lemme egalite esperance de la marche aleatoire conditionnee.} with the application $H \, : \, f \mapsto \|f\|_\infty$  defined on $\cC([0,1],\R)$. Note that $H \in \cF$. This yields
	\begin{align}
		\E \left[ \|{S_{\kappa n}}_{|1-t_n} \|_\infty | S_{\kappa n}(1)=0 \right] & = \E \left[ \|{S_{\kappa n}}_{|1-t_n} \|_\infty \frac{p^{\kappa n}_{t_n}\left( -\oS^+_{\kappa n}(1-t_n) \right)}{p^{\kappa n}_1(0)} \right] \nonumber                                  \\
		                                                                          & \le \E \left[ \|{S_{\kappa n}}_{|1-t_n} \|_\infty \frac{\sup_{a \in \fS_{\kappa n}(t_n)}p^{\kappa n}_{t_n}(a)}{p^{\kappa n}_1(0)} \right].\label{eq: preuve generale majoration A1 2.}
	\end{align}
	Recall that for every $s \in [0,1]$, $p_s$ is the probability density function of the Brownian motion at time $s$. By Lemma \ref{l: lemme convergence des fonctions de probas.}, for every $n \ge 1$, we have
	\begin{equation*}
		\sup_{s \in \left[\frac{1}{2},1\right]} \sup_{a \in \fS_n(s)} \left|\frac{\sqrt{n}}{h} p^n_s(a) - p_s(a)\right| \lesssim \max \left( 1, \tau(\L,n,\delta) \right).
	\end{equation*}
	Furthermore, since $\displaystyle \frac{\sqrt{n}}{h} p^n_1(0) \xrightarrow[n \to \infty]{} p_1(0)$, for every $n$ sufficiently large,
	\begin{equation*}
		\frac{\sqrt{n}}{h} p^{n}_1(0) \ge \frac{p_1(0)}{2}\cdotp
	\end{equation*}
	Thus, for every $n$ sufficiently large such that the inequality just above holds and such that $\displaystyle t_n \ge \frac{1}{2}$, we get
	\begin{align}
		\frac{\sup_{a \in \fS_{\kappa n}(t_n)}p^{\kappa n}_{t_n}\left( a \right)}{p^{\kappa n}_1(0)} & \le \frac{2}{p_1(0)} \sup_{a \in \fS_{\kappa n}(t_n)} \frac{\sqrt{\kappa n}}{h} p^{\kappa n}_{t_n}(a) \nonumber                                                                                                  \\
		                                                                                             & \le \frac{2}{p_1(0)} \left( \sup_{a \in \fS_{\kappa n}(t_n)} \left|\frac{\sqrt{\kappa n}}{h}p^{\kappa n}_{t_n}(a) - p_{t_n} (a)\right| + \sup_{a \in \R} p_{t_n}(a) \right) \nonumber                            \\
		                                                                                             & \lesssim \frac{2}{p_1(0)} \max \left( 1, \tau(\L,n,\delta) \right) + \frac{2}{\sqrt{t_n}} \le \frac{2}{p_1(0)} \max \left( 1, \tau(\L,n,\delta) \right) + 2\sqrt{2}.\label{eq: preuve generale majoration A1 3.}
	\end{align}
	Hence, combining \eqref{eq: preuve generale majoration A1 1.}, \eqref{eq: preuve generale majoration A1 2.} and \eqref{eq: preuve generale majoration A1 3.}, for every $n$ sufficiently large, we obtain the following bound,
	\begin{equation}
		A_1(n,t_n) \lesssim \max \left( 1, \tau(\L,n,\delta) \right) \E \left[ \|{S_{\kappa n}}_{|1-t_n} \|_\infty \right].\label{eq: preuve generale majoration A1 4.}
	\end{equation}
	It remains to bound from above $\E \left[ \|{S_{\kappa n}}_{|1-t_n} \|_\infty \right]$. To this aim, we use that $S_{\kappa n}$ converges to the Brownian motion $B$. Indeed, writing
	\begin{equation*}
		\E \left[ \|{S_{\kappa n}}_{|1-t_n} \|_\infty \right] \le \E \left[ \|{S_{\kappa n}}_{|1-t_n} \|_\infty \right] - \E \left[ \|B_{|1-t_n} \|_\infty \right] + \E \left[ \|B_{|1-t_n} \|_\infty \right],
	\end{equation*}
	we can use Theorem 3.4 in \cite{CouDec2020}. Since the function defined for every $f \in \cC([0,1],\R)$ by $f \mapsto \|f_{|1-t_n}\|_\infty$ is $1$-Lipschitz, we get that
	\begin{multline}
		\E \left[ \|{S_{\kappa n}}_{|1-t_n} \|_\infty \right] - \E \left[ \|B_{|1-t_n} \|_\infty \right] \\ \le \sup_{G \in \mathrm{Lip}_1(\cC([0,1],\R))} \E \left[ G(S_{\kappa n}) \right] - \E \left[ G(B) \right] \lesssim n^{-\frac{1}{6}} \log(n),\label{eq: preuve generale majoration A1 5.}
	\end{multline}
	where
	\begin{multline*}
		\mathrm{Lip}_1 \left(\cC([0,1],\R)\right) = \left\{G \, : \, \cC([0,1],\R) \to \R, \, |G(\omega_1) - G(\omega_2)| \le \|\omega_1-\omega_2\|_\infty, \, \right. \\ \left. \forall \omega_1,\omega_2 \in \cC([0,1],\R)\right\}.
	\end{multline*}
	Following the end of the proof for the bound of $A_3(t_n)$, i.e.\ using the same reasoning as to get \eqref{eq: preuve generale majoration 3.} and the definition of $t_n$ given at \eqref{eq: preuve generale definition tn.}, we get
	\begin{equation}
		\E \left[ \|B_{|1-t_n} \|_\infty \right] = \E \left[ \sup_{s \in [0,1-t_n]} |B_s-B_0| \right] \lesssim (1-t_n)^{\frac{1}{2}} \le n^{-\frac{1}{18}}.\label{eq: preuve generale majoration A1 6.}
	\end{equation}
	Hence, by \eqref{eq: preuve generale majoration A1 4.}, \eqref{eq: preuve generale majoration A1 5.} and \eqref{eq: preuve generale majoration A1 6.}, for all $n \ge 1$,
	\begin{equation*}
		A_1(n,t_n) \lesssim n^{-\frac{1}{18}} \log(n) \max \left( 1, \tau(\L,n,\delta) \right) \lesssim \max\left( n^{-\frac{1}{18}}\log(n), \rho(\L,n), \tau(\L,n,\delta)\right).
	\end{equation*}
	It remains to deal with $A_2(n,t_n)$.
	%\paragraph*{Bound for $A_2(n,t_n)$.}
	Recall the definition of $L_t$ and $L^n_t$ given at \eqref{eq: definition de Lt.} and \eqref{eq: definition de Lnt.}.
	By \eqref{eq: egalite esperance du pont brownien.} and Lemma \ref{l: lemme egalite esperance de la marche aleatoire conditionnee.}, we have the following equalities
	\begin{equation*}
		A_2(n,t_n) = \E \left[  F \left( {S_{\kappa n}}\btn \right) L^n_{t_n}\left(\oS^+_{\kappa n}(t_n)\right)\right] - \E \left[ F \left( B\btn \right) L_{t_n}(B_{t_n}) \right] = \sum_{j=1}^3 A_{2,j}(n,t_n),
	\end{equation*}
	where
	\begin{align*}
		A_{2,1}(n,t_n)             & = \E \left[  F \left( {S_{\kappa n}}\btn \right) L^n_{t_n}\left(\oS^+_{\kappa n}(t_n)\right)\right] - \E \left[  F \left( {S_{\kappa n}}\btn \right) L_{t_n}\left(\oS^+_{\kappa n}(t_n)\right)\right], \\
		A_{2,2}(n,t_n)             & = \E \left[  F \left( {S_{\kappa n}}\btn \right) L_{t_n}\left(\oS^+_{\kappa n}(t_n)\right)\right] - \E \left[  F \left( {S_{\kappa n}}\btn \right) L_{t_n}\left(S_{\kappa n}(t_n)\right)\right],       \\
		\text{and } A_{2,3}(n,t_n) & = \E \left[  F \left( {S_{\kappa n}}\btn \right) L_{t_n}\left(S_{\kappa n}(t_n)\right)\right] - \E \left[  F \left( B\btn \right) L_{t_n}\left(B_{t_n}\right) \right].
	\end{align*}
	The end of the proof consists in bounding from above each $A_{2,j}(n,t_n)$.
	First, since $\|F\|_\infty \le 1$,  we have
	\begin{equation*}
		A_{2,1}(n,t_n) \le \E \left[ \left| L^n_{t_n}\left(\oS^+_{\kappa n}(t_n)\right) - L_{t_n}\left(\oS^+_{\kappa n}(t_n)\right) \right| \right].
	\end{equation*}
	Using that for all $n \ge 1$, $\displaystyle t_n \in \left[0, 1-n^{-\frac{1}{9}}\right]$, we can apply Lemma \ref{l: lemme implication des hypotheses.}. We get that for all $n \ge 1$, %(recall that for $n$ sufficiently large, $\displaystyle n^{-\frac{1}{18}}\log(n) \ge n^{-\frac{1}{18}}$),
	\begin{equation*}
		A_{2,1}(n,t_n) \lesssim \max\left( n^{-\frac{1}{18}}\log(n), \rho(\L,n), \tau(\L,n,\delta)\right).
	\end{equation*}
	Then, for all $n \ge 1$ and for all $x,y \in \R$,
	\begin{equation}
		\left| L_{t_n} (x) - L_{t_n} (y) \right| \le n^{\frac{1}{9}} |x-y|.\label{c: claim fonctions Ltn.}
	\end{equation}
	This is an immediate consequence of the mean value inequality using that for all $x \in \R$,
	\begin{equation*}
		L'_{t_n} (x) = \frac{-x}{(1-t_n)^{\frac{3}{2}}} \exp \left( \frac{-x^2}{2(1-t_n)} \right),
	\end{equation*}
	and then
	\begin{equation*}
		\sup_{x \in \R} \left| L'_{t_n} (x) \right| = \exp \left( - \frac{1}{2} \right) \frac{1}{1-t_n} \le n^{\frac{1}{9}}.
	\end{equation*}
	By the fact that $\|F\|_\infty \le 1$ and by \eqref{c: claim fonctions Ltn.},
	\begin{equation}
		A_{2,2}(n,t_n) \le n^{\frac{1}{9}} \E \left[ \left| \oS^+_{\kappa n} (t_n) - S_{\kappa n} (t_n) \right| \right].\label{eq: preuve generale A22.}
	\end{equation}
	By the definitions of $\oS_{\kappa n}$ and $S_{\kappa n}$,
	\begin{equation}
		\left| \oS^+_{\kappa n} (t_n) - S_{\kappa n} (t_n) \right| \le \frac{1}{\sqrt{n}} \left| X_{\ent + 1} \right|.\label{eq: preuve generale A22 2.}
	\end{equation}
	Since the $X_i$ are i.i.d. random variables, combining \eqref{eq: preuve generale A22.} and \eqref{eq: preuve generale A22 2.} yields
	\begin{equation*}
		A_{2,2}(n,t_n) \le n^{-\frac{7}{18}} \E \left[ |X_1| \right].
	\end{equation*}
	The assumptions on the moments of $\L$ in Theorem \ref{thm: theorem principal general.} ensure that $\E \left[ |X_1| \right] < \infty$. Therefore, noting that for $n$ sufficiently large, $\displaystyle n^{-\frac{7}{18}} \le n^{-\frac{1}{18}} \log(n)$, we get that for all $n \ge 1$,
	\begin{equation*}
		A_{2,2}(n,t_n) \lesssim n^{-\frac{1}{18}}\log(n) \le \max\left( n^{-\frac{1}{18}}\log(n), \rho(\L,n), \tau(\L,n,\delta)\right).
	\end{equation*}
	Finally, to deal with $A_{2,3}(n,t_n)$, define for every $n \ge 1$,
	%\subparagraph*{Bound for $A_{2,3}(n,t_n)$.}\label{subp: l'utilisation du theoreme se degrade en approchant de 1.}
	\begin{equation*}
		\begin{array}{rcl}
			G_n \, : \, \cC([0,1],\R) & \to     & \R                                                                                \\
			\omega                    & \mapsto & \frac{1}{2} n^{-\frac{1}{9}} F \left( \omega_{|t_n} \right) L_{t_n}(\omega(t_n)).
		\end{array}
	\end{equation*}
	Then, for every $\omega_1,\omega_2 \in \cC([0,1],\R)$,
	\begin{multline*}
		|G_n(\omega_1) - G_n(\omega_2)|
		= \frac{1}{2} n^{-\frac{1}{9}}
		\left|
		F \left( \omega_{1|t_n} \right) L_{t_n}(\omega_1(t_n))
		- F \left( \omega_{2|t_n} \right) L_{t_n}(\omega_2(t_n))
		\right| \\
		\le \frac{1}{2} n^{-\frac{1}{9}}
		\Big[
			\left| F \left( \omega_{1|t_n} \right) \right|
			\left| L_{t_n}(\omega_1(t_n)) - L_{t_n}(\omega_2(t_n)) \right| \\
			\qquad\qquad
			+ \left| L_{t_n}(\omega_2(t_n)) \right|
			\left| F \left( \omega_{1|t_n} \right)
			- F \left( \omega_{2|t_n} \right) \right|
			\Big].
	\end{multline*}
	Now, it is clear that
	$\displaystyle \left| F \left( \omega_{1|t_n} \right) \right| \le 1$ since $\|F\|_\infty \le 1$. By
	\eqref{c: claim fonctions Ltn.}, we have
	\begin{equation*}
		\left| L_{t_n}(\omega_1(t_n)) - L_{t_n}(\omega_2(t_n)) \right| \le n^{\frac{1}{9}} \left| \omega_1(t_n) - \omega_2(t_n) \right| \le n^{\frac{1}{9}} \| \omega_1 - \omega_2 \|_\infty.
	\end{equation*}
	Moreover, the inequality
	\begin{equation*}
		\|L_{t_n}\|_\infty = \frac{1}{\sqrt{1-t_n}} \le n^{\frac{1}{18}},
	\end{equation*}
	induces  $\displaystyle \left| L_{t_n}(\omega_2(t_n)) \right| \le n^{\frac{1}{18}}$.
	Since  $F$ is $1$-Lipschitz, we have the bound:
	\begin{equation*}
		\left| F \left( \omega_{1|t_n} \right) - F \left( \omega_{2|t_n} \right) \right| \le \| \omega_{1|t_n} - \omega_{2|t_n} \|_\infty \le \|\omega_1-\omega_2\|_\infty.
	\end{equation*}
	Hence, it follows that
	\begin{equation*}
		|G_n(\omega_1) - G_n(\omega_2)| \le \frac{1}{2} n^{-\frac{1}{9}} \left( n^{\frac{1}{9}} \| \omega_1 - \omega_2 \|_\infty + n^{\frac{1}{18}} \| \omega_1 - \omega_2 \|_\infty \right) \le \| \omega_1 - \omega_2 \|_\infty,
	\end{equation*}
	and $G_n$ is $1$-Lipschitz.
	It allows us to use Theorem 3.4 in \cite{CouDec2020}.
	By this theorem, for every $n \ge 1$,
	\begin{align*}
		\E \left[G_n \left( S_{\kappa n} \right) \right] - \E \left[G_n \left( B \right) \right] & \le \sup_{G \in \mathrm{Lip}_1(\cC([0,1],\R))} \E \left[ G(S_{\kappa n}) \right] - \E \left[ G(B) \right] \lesssim n^{-\frac{1}{6}} \log(n).
	\end{align*}
	To conclude, we just have to introduce the factor $\displaystyle 2n^{\frac{1}{9}}$ to get
	\begin{align*}
		A_{2,3}(n,t_n) = 2n^{\frac{1}{9}} \left( \E \left[G_n \left( S_{\kappa n} \right) \right] - \E \left[G_n \left( B \right) \right] \right) & \lesssim n^{-\frac{1}{18}} \log(n)                                                  \\
		                                                                                                                                          & \lesssim \max\left( n^{-\frac{1}{18}}\log(n), \rho(\L,n), \tau(\L,n,\delta)\right).
	\end{align*}
\end{proof}

\begin{proof}[Proof of Lemma \ref{c: claim pour majorer Aunntn.}.]
	For every $n \ge 1$, recall that $\displaystyle t_n = 1 - n^{-\frac{1}{9}}$ and let
	\begin{equation*}
		[t_n] = \frac{\lceil t_n \kappa n \rceil}{\kappa n}\cdotp
	\end{equation*}
	For every $x \in \R$, let $\{x\} = x - \lfloor x \rfloor$.
	Let $s \in [0,1]$. If $s \le t_n$ then
	\begin{equation*}
		S_{\kappa n} (s) - {S_{\kappa n}}_{|t_n} (s) = 0.
	\end{equation*}
	If $\displaystyle s \in \left[ t_n, [t_n] \right]$ then
	\begin{equation*}
		\left| S_{\kappa n} (s) - {S_{\kappa n}}_{|t_n} (s) \right| = \frac{\{\kappa n s\} - \{\kappa n t_n\}}{\sqrt{\kappa n}} \left| X_{\lfloor \kappa n t_n \rfloor + 1} \right| \le \frac{\{\kappa n (1-t_n)\}}{\sqrt{\kappa n}} \left| X_{\lfloor \kappa n t_n \rfloor + 1} \right|.
	\end{equation*}
	If $\displaystyle s > [t_n]$ then
	\begin{align*}
		\left| S_{\kappa n} (s) - {S_{\kappa n}}_{|t_n} (s) \right| & \le \left| S_{\kappa n} (s) - S_{\kappa n} \left( [t_n] \right) \right| + | S_{\kappa n} \left( [t_n] \right) - S_{\kappa n} (t_n) |                                                                                                                                                         \\
		                                                            & \le \left| \frac{1}{\sqrt{\kappa n}} \sum_{i = \kappa n [t_n] + 1}^{\lfloor \kappa n s \rfloor} X_i + \frac{\{ \kappa n s \}}{\sqrt{\kappa n}} X_{\lfloor \kappa n s \rfloor + 1}\right| + \frac{\{\kappa n (1-t_n)\}}{\sqrt{\kappa n}} \left| X_{\lfloor \kappa n t_n \rfloor + 1} \right|.
	\end{align*}
	Thus, we deduce that
	\begin{multline*}
		\sup_{s \in [0,1]} \left| S_{\kappa n} (s) - {S_{\kappa n}}_{|t_n} (s) \right| \\ \le \sup_{s \in \left[ [t_n], 1 \right]} \left| \frac{1}{\sqrt{\kappa n}} \sum_{i = \kappa n [t_n] + 1}^{\lfloor \kappa n s \rfloor} X_i + \frac{\{ \kappa n s \}}{\sqrt{\kappa n}} X_{\lfloor \kappa n s \rfloor + 1}\right| + \frac{\{\kappa n (1-t_n)\}}{\sqrt{\kappa n}} \left| X_{\lfloor \kappa n t_n \rfloor + 1} \right|.
	\end{multline*}
	Now, since $(X_i, \, i \ge 1)$ is a sequence of i.i.d. random variables, given that $S_{\kappa n}(1)=0$,
	\begin{multline*}
		\sup_{s \in [[t_n], 1]} \Bigg|
		\frac{1}{\sqrt{\kappa n}} \sum_{i = \kappa n [t_n] + 1}^{\lfloor \kappa n s \rfloor} X_i
		+ \frac{\{\kappa n s\}}{\sqrt{\kappa n}} X_{\lfloor \kappa n s \rfloor + 1}
		\Bigg| \\
		\text{and} \quad
		\sup_{s \in [[t_n], 1]} \Bigg|
		\frac{1}{\sqrt{\kappa n}} \sum_{i = 1}^{\lfloor \kappa n s \rfloor - \kappa n [t_n]} X_i
		+ \frac{\{\kappa n s\}}{\sqrt{\kappa n}} X_{\lfloor \kappa n s \rfloor - \kappa n [t_n] + 1}
		\Bigg|
	\end{multline*}
	have the same distribution. Similarly, the random variables $\displaystyle X_{\lfloor \kappa n t_n \rfloor + 1}$ and $X_1$ have the same distribution given that $S_{\kappa n}(1)=0$.
	Hence,
	\begin{multline}
		\E \left[ \|S_{\kappa n} - {S_{\kappa n}}_{|t_n} \|_\infty \Bigl | S_{\kappa n}(1)=0 \right]
		=   \E \left[ \left. \sup_{s \in [0,1]} \left| S_{\kappa n} - {S_{\kappa n}}_{|t_n}
		\right| \right| S_{\kappa n}(1)=0 \right] \\
		\shoveleft{\le  \E \left[ \left. \sup_{s \in \left[ [t_n], 1 \right]}
				\left| \frac{1}{\sqrt{\kappa n}} \sum_{i = \kappa n [t_n] + 1}^{\lfloor \kappa n s \rfloor}
				X_i + \frac{\{ \kappa n s \}}{\sqrt{\kappa n}} X_{\lfloor \kappa n s \rfloor + 1}\right| \right|
				S_{\kappa n}(1)=0 \right]}		\\
		\shoveright{+
			\E \left[ \left. \frac{\{\kappa n (1-t_n)\}}{\sqrt{\kappa n}} \left| X_{\lfloor \kappa n t_n \rfloor + 1}
				\right| \right| S_{\kappa n}(1)=0 \right]
		}\\
		\shoveleft{\le \E \left[ \left. \sup_{s \in \left[ [t_n], 1 \right]}
				\left| \frac{1}{\sqrt{\kappa n}} \sum_{i = 1}^{\lfloor \kappa n s \rfloor -
						\kappa n [t_n]} X_i + \frac{\{ \kappa n s \}}{\sqrt{\kappa n}} X_{\lfloor \kappa n s \rfloor -
						\kappa n [t_n] + 1} \right| \right| S_{\kappa n}(1)=0 \right]
		}\\
		+ \E \left[ \left. \frac{\{\kappa n (1-t_n)\}}{\sqrt{\kappa n}} \left| X_1 \right| \right| S_{\kappa n}(1)=0 \right].\label{eq: preuve majoration A un claim technique 1.}
	\end{multline}
	On the one hand, %since $\kappa n [t_n] \in \N$, 
	%\begin{equation*}
	%	\lfloor \kappa n s \rfloor - \kappa n [t_n] = \left\lfloor \kappa n \left( s - [t_n] \right) \right\rfloor,
	%\end{equation*}
	%and
	%\begin{equation*}
	%	\{ \kappa n s \} = \left\{ \kappa n s - \kappa n [t_n] \right\} = \left\{ \kappa n (s-[t_n]) \right\}. 
	%\end{equation*}
	%This gives
	\begin{multline*}
		\sup_{s \in [[t_n], 1]} \Bigg|
		\frac{1}{\sqrt{\kappa n}} \sum_{i = 1}^{\lfloor \kappa n s \rfloor - \kappa n [t_n]} X_i
		+ \frac{\{\kappa n s\}}{\sqrt{\kappa n}} X_{\lfloor \kappa n s \rfloor - \kappa n [t_n] + 1}
		\Bigg| \\
		= \sup_{s \in [0, 1-[t_n]]} \Bigg|
		\frac{1}{\sqrt{\kappa n}} \sum_{i = 1}^{\lfloor \kappa n s \rfloor} X_i
		+ \frac{\{\kappa n s\}}{\sqrt{\kappa n}} X_{\lfloor \kappa n s \rfloor + 1}
		\Bigg| \\
		\le \sup_{s \in [0, 1-t_n]} \Bigg|
		\frac{1}{\sqrt{\kappa n}} \sum_{i = 1}^{\lfloor \kappa n s \rfloor} X_i
		+ \frac{\{\kappa n s\}}{\sqrt{\kappa n}} X_{\lfloor \kappa n s \rfloor + 1}
		\Bigg|
		= \|S_{\kappa n | 1-t_n}\|_\infty,
	\end{multline*}
	where the inequality comes from the fact that $1-[t_n] \le 1-t_n$.
	This gives
	\begin{multline}
		\E \left[ \left. \sup_{s \in \left[ [t_n], 1 \right]}
			\left| \frac{1}{\sqrt{\kappa n}} \sum_{i = 1}^{\lfloor \kappa n s
					\rfloor - \kappa n [t_n]} X_i + \frac{\{ \kappa n s \}}{\sqrt{\kappa n}}
			X_{\lfloor \kappa n s \rfloor - \kappa n [t_n] + 1} \right| \right| S_{\kappa n}(1)=0 \right] \\
		\le \E \left[ \left. \|S_{\kappa n | 1-t_n}\|_\infty \right| S_{\kappa n}(1)=0 \right].\label{eq: preuve majoration A un claim technique 2.}
	\end{multline}
	On the other hand,
	\begin{equation}
		\frac{\{\kappa n (1-t_n)\}}{\sqrt{\kappa n}} \left| X_1 \right| \le \frac{1}{\sqrt{\kappa n}} |X_1| = \left| S_{\kappa n} \left( \frac{1}{\kappa n} \right) \right|.\label{eq: preuve majoration A un claim technique 3.}
	\end{equation}
	Since $\displaystyle \frac{1}{1 - t_n} = n^{\frac{1}{9}} \le \kappa n$, we get that $\displaystyle \frac{1}{\kappa n} \in [0,1-t_n]$ and thus
	\begin{equation}
		\left| S_{\kappa n} \left( \frac{1}{\kappa n} \right) \right| \le \sup_{s \in [0,1-t_n]} \left| S_{\kappa n} (s) \right| = \|S_{\kappa n | 1-t_n}\|_\infty.\label{eq: preuve majoration A un claim technique 4.}
	\end{equation}
	Combining \eqref{eq: preuve majoration A un claim technique 3.} and \eqref{eq: preuve majoration A un claim technique 4.}, we get
	\begin{equation}
		\E \left[ \left. \frac{\{\kappa n (1-t_n)\}}{\sqrt{\kappa n}} \left| X_1 \right| \right| S_{\kappa n}(1)=0 \right] \le \E \left[ \left. \|S_{\kappa n | 1-t_n}\|_\infty \right| S_{\kappa n}(1)=0 \right].\label{eq: preuve majoration A un claim technique 5.}
	\end{equation}
	It allows us to conclude since \eqref{eq: claim pour majorer Aunntn.} holds for every $n \ge 1$ by combining \eqref{eq: preuve majoration A un claim technique 1.}, \eqref{eq: preuve majoration A un claim technique 2.} and \eqref{eq: preuve majoration A un claim technique 5.}.
\end{proof}

\begin{proof}[Proof of Lemma \ref{l: lemme convergence des fonctions de probas.}.]
	Let $\displaystyle \delta > \frac{1}{18}$ and $\eta < \min(2\delta,1)$.
	In this proof, we write $a \lesssim b$ if there exists a constant $C$ depending only on $\L$, $\delta$ and $\eta$ such that $a \le C b$.
	In \cite{Durrett2019}, Theorem 3.5.3 gives that
	\begin{equation}
		\sup_{a \in \fS_n(s)} \left| \frac{\sqrt{n}}{h} p^n_s(a) - p_s(a) \right| \xrightarrow[n \to \infty]{}0. \label{eq: resultat de durrett.}
	\end{equation}
	Here, we follow the proof of this theorem and its ideas to check that we can control the rate of convergence in \eqref{eq: resultat de durrett.}.
	Suppose that the assumptions of Lemma \ref{l: lemme convergence des fonctions de probas.} are satisfied.
	For every $n \ge 1$, denote $\displaystyle I_n = \left[ n^{- \eta}, 1 \right]$.
	For every $n \ge 1$ and every $s \in I_n$, denote by $\psi_{s,n}$ the characteristic function of $\displaystyle \frac{1}{\sqrt{n}} \sum_{i=1}^{\ens} X_i = x$.
	Then, for every $u \in \R$,
	\begin{equation*}
		\psi_{s,n}(u) = \E \left[ \exp \left( \frac{iu}{\sqrt{n}} \sum_{j=1}^{\ens} X_j \right) \right] = \phi^{\ens} \left( \frac{u}{\sqrt{n}} \right),
	\end{equation*}
	where $\phi$ denotes the characteristic function of $\L$.
	Following the explanations in \cite{Durrett2019}, we get for every $a$ in $\fS_n(s)$,
	\begin{equation}
		\frac{\sqn}{h} p^n_s(a) = \frac{1}{2\pi} \int_{-\pi\sqrt{n}/h}^{\pi\sqrt{n}/h} \mathrm{e}^{-iua} \phi^{\ens} \left( \frac{u}{\sqrt{n}} \right) \mathrm{d}u.\label{eq: claim section 2.1 1.}
	\end{equation}
	On the other hand, since for a standard Brownian motion $B$, $B(s)$ is a centered Gaussian random variable of variance $s$, using the inversion formula for characteristic functions,
	\begin{equation}
		p_s(a) = \frac{1}{2\pi} \int_\R \mathrm{e}^{-iua} \exp \left( -\frac{su^2}{2} \right) \mathrm{d}u.\label{eq: claim section 2.1 2.}
	\end{equation}
	Now, subtracting \eqref{eq: claim section 2.1 1.} and \eqref{eq: claim section 2.1 2.} as in \cite{Durrett2019}, we get (recall that $\pi>1$ and $\left|\mathrm{e}^{-iua}\right| \le 1$)
	\begin{equation*}
		\left| \frac{\sqn}{h} p^n_s(a) - p_s(a) \right| \le \int_{-\pi\sqn/h}^{\pi\sqn/h} \left| \phi^{\ens} \left( \frac{u}{\sqrt{n}} \right) - \exp \left( -\frac{su^2}{2} \right) \right| \mathrm{d}u + \int_{\pi\sqn/h}^{\infty} \exp \left( -\frac{su^2}{2} \right) \mathrm{d}u.
	\end{equation*}
	In the inequality above, note that the upper bound is independent of $a$.
	Furthermore, since there exists $\epsilon>0$ such that $\displaystyle n^\epsilon \lesssim su^2$ if $\displaystyle u \ge \frac{\pi \sqrt{n}}{h}$ and $s \in I_n$ (recall that $\eta<1$ and $\displaystyle s \ge n^{-\eta}$), it is easy to see that for every $n \ge 1$ and every $s \in I_n$,
	\begin{equation*}
		\int_{\pi\sqn/h}^{\infty} \exp \left( -\frac{su^2}{2} \right) \mathrm{d}u \lesssim n^{- \frac{1}{18}} \le \max \left( n^{- \frac{1}{18}}, \tau(\L,n,\delta) \right).
	\end{equation*}
	Then, to conclude, it remains to prove that for every $n \ge 1$ and every $s \in I_n$,
	\begin{equation}
		A(n,s) := \int_{-\pi\sqn/h}^{\pi\sqn/h} \left| \phi^{\ens} \left( \frac{u}{\sqrt{n}} \right) - \exp \left( -\frac{su^2}{2} \right) \right| \mathrm{d}u \lesssim \max \left( n^{- \frac{1}{18}}, \tau(\L,n,\delta) \right). \label{eq: claim section 2.1 3.}
	\end{equation}
	Thanks to the reasoning given in the proof of Theorem 3.5.3 in \cite{Durrett2019}, and since $\E [X_1] =0$ and $\E [X^2_1] = 1$, there exists $\gamma \in (0,\pi)$ such that for all $u \in [-\gamma\sqn,\gamma\sqn]$,
	\begin{equation*}
		\left| \phi^{\ens} \left( \frac{u}{\sqrt{n}} \right) \right| \le \exp \left( -\frac{su^2}{4} \right).
	\end{equation*}
	Note $\displaystyle \delta' = \min \left(\delta,\frac{1}{2}\right)$.
	Dividing the integral in \eqref{eq: claim section 2.1 3.} into three pieces, we get $\displaystyle A(n,s) \le \sum_{j=1}^3 A_j(n,s)$ where
	\begin{multline*}
		A_1(n,s) = \int_{-n^{\delta'}}^{n^{\delta'}} \left| \phi^{\ens} \left( \frac{u}{\sqrt{n}} \right) - \exp \left( -\frac{su^2}{2} \right) \right| \mathrm{d}u, \\
		\shoveleft{A_2(n,s) = 2 \int_{n^{\delta'}}^{\gamma\sqn} \exp \left( - \frac{su^2}{4} \right) \mathrm{d}u,} \\
		\shoveleft{A_3(n,s) = \int_{-\pi\sqn/h}^{-\gamma\sqn} \left| \phi^{\ens} \left( \frac{u}{\sqrt{n}} \right) - \exp \left( -\frac{su^2}{2} \right) \right| \mathrm{d}u} \\ + \int_{\gamma\sqn}^{\pi\sqn/h} \left| \phi^{\ens} \left( \frac{u}{\sqrt{n}} \right) - \exp \left( -\frac{su^2}{2} \right) \right| \mathrm{d}u.
	\end{multline*}
	Since $\displaystyle \delta' = \min \left(\delta,\frac{1}{2}\right)$, we have
	\begin{equation*}
		A_1(n,s) \le \tau(\L,n,\delta) \le \max \left( n^{- \frac{1}{18}}, \tau(\L,n,\delta) \right).
	\end{equation*}
	Furthermore, by the definitions of $\eta$ and $\delta'$, it is clear that $\eta < 2 \delta'$. It implies that there exists $\epsilon'>0$ such that $\displaystyle su^2 \ge n^{\epsilon'}$ whenever $\displaystyle u \ge n^{\delta'}$ and $s \in I_n$. Then, it follows that for all $n \ge 1$ and $s \in I_n$,
	\begin{equation*}
		A_2(n,s) \lesssim n^{-\frac{1}{18}} \le \max \left( n^{- \frac{1}{18}}, \tau(\L,n,\delta) \right).
	\end{equation*}
	Finally, as in \cite{Durrett2019}, using that the $X_j$ are random variables with lattice distribution with span $h$, there is $\xi \in (0,1)$ such that $|\phi(u)| \le \xi < 1$ for $|u| \in [\gamma,\pi/h]$. Thus,
	\begin{equation*}
		\left| \phi \left( \frac{u}{\sqrt{n}} \right) \right| \le \xi < 1 \text{ for } u \in [\gamma\sqn,\pi\sqn/h].
	\end{equation*}
	Hence, for every $n \ge 1$ and $s \in I_n$, using again that $\eta < 1$, we obtain
	\begin{equation*}
		A_3(n,s) \le 2 \int_{\gamma \sqn}^{\pi\sqn/h} \xi^{\ens} + \exp \left( -\frac{su^2}{2} \right) \mathrm{d}u \lesssim n^{-\frac{1}{18}} \le \max \left( n^{- \frac{1}{18}}, \tau(\L,n,\delta) \right).
	\end{equation*}
\end{proof}

\begin{proof}[Proof of Lemma \ref{l: lemme implication des hypotheses.}.]
	Assume that $\L$ satisfies the assumptions of Theorem \ref{thm: theorem principal general.}.
	Let $\displaystyle \delta > \frac{1}{18}$.
	By the triangle inequality, for every $n \ge 1$ and $\displaystyle t \in \left[ 0, 1-n^{-\frac{1}{9}} \right]$, for every $a \in - \fS_{\kappa n}(1-t)$,
	\begin{align*}
		|L_t(a)-L^n_t(a)| & = \left|\frac{p_{1-t}(-a)}{p_1(0)} - \frac{p^{\kappa n}_{1-t}(-a)}{p^{\kappa n}_1(0)}\right|                                                                                                                                                                                    \\
		                  & \le \left|\frac{p_{1-t}(-a)}{p_1(0)} - \frac{p_{1-t}(-a)}{\frac{\sqrt{\kappa n}}{h}p^{\kappa n}_1(0)}\right| + \left|\frac{p_{1-t}(-a)}{\frac{\sqrt{\kappa n}}{h}p^{\kappa n}_1(0)} - \frac{p^{\kappa n}_{1-t}(-a)}{p^{\kappa n}_1(0)},\right|                                  \\
		                  & \le \|p_{1-t}\|_\infty \frac{\left|\frac{\sqrt{\kappa n}}{h}p^{\kappa n}_1(0)-p_1(0)\right|}{\frac{\sqrt{\kappa n}}{h}p^{\kappa n}_1(0)p_1(0)} + \frac{\left|\frac{\sqrt{\kappa n}}{h}p^{\kappa n}_{1-t}(-a) - p_{1-t}(-a)\right|}{\frac{\sqrt{\kappa n}}{h}p^{\kappa n}_1(0)}.
	\end{align*}
	Since $\displaystyle \frac{\sqrt{\kappa n}}{h}p^{\kappa n}_1(0) \xrightarrow[n \to \infty]{}p_1(0)$, we have for $n$ sufficiently large,
	\begin{equation*}
		\frac{\sqrt{\kappa n}}{h}p^{\kappa n}_1(0) \ge \frac{p_1(0)}{2} = \frac{1}{2\sqrt{2\pi}}\cdotp
	\end{equation*}
	Furthermore, for $\displaystyle t \le 1-n^{-\frac{1}{9}}$, the following inequality holds
	\begin{equation*}
		\|p_{1-t}\|_\infty = \frac{1}{\sqrt{2 \pi (1-t)}} \le \frac{n^{\frac{1}{18}}}{\sqrt{2\pi}}\cdotp
	\end{equation*}
	Thus, we get for any $n$ sufficiently large and $\displaystyle t \in \left[ 0, 1-n^{-\frac{1}{9}} \right]$,
	\begin{multline*}
		\sup_{a \in -\fS_{\kappa n}(1-t)} |L_t(a)-L^n_t(a)| \\
		\le 2 \sqrt{2 \pi} \Bigg(
		n^{\frac{1}{18}} \left| \frac{\sqrt{\kappa n}}{h} p^{\kappa n}_1(0) - p_1(0) \right|
		+ \sup_{s \in \left[ n^{-\frac{1}{9}}, 1 \right]} \sup_{a \in \fS_{\kappa n}(s)}
		\left| \frac{\sqrt{\kappa n}}{h} p^{\kappa n}_s(a) - p_s(a) \right|
		\Bigg).
	\end{multline*}
	The assumptions of Lemma \ref{l: lemme convergence des fonctions de probas.} are satisfied.
	Since $\displaystyle \delta > \frac{1}{18}$, the conclusion of Lemma \ref{l: lemme convergence des fonctions de probas.} holds for $\eta = \frac{1}{9} < \min(2\delta,1)$. This gives a constant $C>0$ such that for all $n \ge 1$,
	\begin{equation*}
		\sup_{s \in \left[ n^{-\frac{1}{9}}, 1 \right]} \sup_{a \in \fS_{\kappa n}(s)} \left|\frac{\sqrt{\kappa n}}{h} p^{\kappa n}_s(a) - p_s(a)\right| \le C \max \left( n^{-\frac{1}{18}}, \tau(\L,n,\delta) \right).
	\end{equation*}
	Noting that, by definition,
	\begin{equation*}
		n^{\frac{1}{18}} \left| \frac{\sqrt{\kappa n}}{h} p^{\kappa n}_1(0) - p_1(0) \right| = \rho (\L,n),
	\end{equation*}
	we conclude the first part of the proof of Lemma \ref{l: lemme implication des hypotheses.}.
	It remains to show that \eqref{eq: hypothese esperance de la difference des fonctions L et Ln.} holds.
	Assume that for every $n \ge 1$ and $\displaystyle t \in \left[ 0, 1 \right]$,
	\begin{equation}
		-\oS^+_{\kappa n}(t) \in \fS_{\kappa n}(1-t) \text{ a.s.}.\label{eq: claim pour ne pas avoir d'indicatrice.}
	\end{equation}
	Then we get that
	\begin{equation*}
		\left| L_t\left(\oS^+_{\kappa n}(t)\right) - L^n_t\left(\oS^+_{\kappa n}(t)\right) \right| \le \sup_{a \in -\fS_{\kappa n}(1-t)} |L_t(a)-L^n_t(a)| \text{ a.s.},
	\end{equation*}
	which proves \eqref{eq: hypothese esperance de la difference des fonctions L et Ln.} thanks to \eqref{eq: lemme implication des hypotheses.}.
	Hence, the aim of the end of the proof is to prove \eqref{eq: claim pour ne pas avoir d'indicatrice.}.
	Let $n \ge 1$ and $t \in [0,1]$.
	By definition of $\kappa$, $\displaystyle \P \left( S_{\kappa n}(1) = 0 \right) > 0$.
	This implies that
	\begin{equation*}
		0 \in \fS_{\kappa n} (1) = \left\{ \frac{\kappa n b + hz}{\sqrt{\kappa n}}, \, z \in \Z \right\}.
	\end{equation*}
	Thus, there exists $z_0 \in \Z$ such that
	\begin{equation}
		hz_0 = -\kappa n b.\label{eq: preuve claim pour ne pas avoir d'indicatrice.}
	\end{equation}
	By definition of $\oS^+_{\kappa n}(t)$,
	\begin{equation*}
		- \oS^+_{\kappa n}(t) \in \left\{ \frac{- \lceil \kappa n t \rceil b + hz}{\sqrt{\kappa n}}, \, z \in \Z \right\} \text{ a.s.}
	\end{equation*}
	Then, it is easy to notice that
	\begin{equation}
		\left\{ \frac{- \lceil \kappa n t \rceil b + hz}{\sqrt{\kappa n}}, \, z \in \Z \right\} \subset \fS_{\kappa n}(1-t),\label{eq: preuve claim pour ne pas avoir d'indicatrice 2.}
	\end{equation}
	which concludes the proof.
	To prove \eqref{eq: preuve claim pour ne pas avoir d'indicatrice 2.}, let us take $\displaystyle x \in \left\{ \frac{- \lceil \kappa n t \rceil b + hz}{\sqrt{\kappa n}}, \, z \in \Z \right\}$. There exists $z_x \in \Z$ such that
	\begin{equation*}
		x = \frac{- \lceil \kappa n t \rceil b + hz_x}{\sqrt{\kappa n}}.
	\end{equation*}
	Since $ - \lceil \kappa n t \rceil = \lfloor \kappa n (1-t) \rfloor - \kappa n$ and using \eqref{eq: preuve claim pour ne pas avoir d'indicatrice.},
	\begin{equation*}
		x = \frac{\lfloor \kappa n (1-t) \rfloor b - \kappa n b + hz_x}{\sqrt{\kappa n}} = \frac{\lfloor \kappa n (1-t) \rfloor b + h(z_0+z_x)}{\sqrt{\kappa n}} \in \fS_{\kappa n}(1-t).
	\end{equation*}
\end{proof}

\subsection{Proof of Theorem \ref{thm: theoreme principal rademacher.}}\label{s: sous-section preuve Rademacher.}

\begin{proof}[Proof of Theorem \ref{thm: theoreme principal rademacher.}]
	In this proof, we write $a \lesssim b$ if there exists $C>0$ which does not depend on $n$ such that $a \le C b$.
	Assume that $\L$ is the Rademacher distribution. Then, $\L$ is a lattice distribution and, if $X$ is a random variable with distribution $\L$, $\E [X] = 0$ and $\displaystyle \E \left[ X^2 \right] = 1$.
	Thus, one can apply Theorem \ref{thm: theorem principal general.} and, to prove Theorem \ref{thm: theoreme principal rademacher.}, it remains to control $\rho(\L,n)$ and $\tau(\L,n,\delta)$ for some $\displaystyle \delta > \frac{1}{18}$.
	First, $\L$ has a span $2$, $\kappa(\L) = 2$ and
	\begin{equation*}
		p^{2n}_1(0) = \binom{2n}{n} \frac{1}{2^{2n}}\cdotp
	\end{equation*}
	The asymptotic development of the central binomial coefficient is well known:
	\begin{equation*}
		\binom{2n}{n} = \frac{4^n}{\sqrt{\pi n}} \left( 1 - \frac{1}{8n} + o \left( \frac{1}{n}\right) \right).
	\end{equation*}
	This gives
	\begin{equation*}
		\left| \frac{\sqrt{\kappa n}}{h} p^{\kappa n}_1(0) - p_1(0) \right| = \left| \frac{\sqrt{n}}{\sqrt{2}} \binom{2n}{n} \frac{1}{2^{2n}} - \frac{1}{\sqrt{2\pi}}\right| \lesssim n^{-1},
	\end{equation*}
	and for all $n \ge 1$,
	\begin{equation}
		\rho(\L,n) \lesssim n^{-\frac{17}{18}}.\label{eq: rade1.}
	\end{equation}
	Furthermore, denote by $\phi$ the characteristic function of $\L$. For every $u \in \R$, $\phi(u) = \cos(u)$.
	We get for every $n \in \N^*$ and $u \in \R$,
	\begin{align*}
		\phi^{\ens} \left( \frac{u}{\sqn} \right) & = \left( \cos \left( \frac{u}{\sqn} \right) \right)^{\ens}                                                                              \\
		%& = \exp \left( \ens \ln \left( \cos \left( \frac{u}{\sqn} \right) \right) \right) \\
		                                          & = \exp \left( \ens \left[ - \frac{u^2}{2n} - \frac{u^4}{12 n^2} + \frac{u^4}{n^2}\epsilon_1 \left(\frac{u}{\sqn}\right) \right] \right) \\
		%& = \exp \left( - \frac{s u^2}{2} \right) \exp \left( -\frac{su^4}{12n} + \frac{\{ns\}u^2}{2n} + \frac{su^4}{n} \epsilon_2 \left(\frac{u}{\sqn}\right) + \frac{u^4}{n} \epsilon_3 \left( \frac{u}{\sqn}, \frac{1}{n} \right) \right),
	\end{align*}
	where $\epsilon_1$ is a continuous function such that $\displaystyle \epsilon_1(x) \xrightarrow[x \to 0]{} 0$.
	%$\{x\} = x - \lfloor x \rfloor$, and $\epsilon_1$, $\epsilon_2$ and $\epsilon_3$ are continuous functions such that $\displaystyle \epsilon_1(x) \xrightarrow[x \to 0]{} 0$,  $\displaystyle \epsilon_2(x) \xrightarrow[x \to 0]{} 0$ and $\displaystyle \epsilon_3(x,y) \xrightarrow[\|(x,y)\| \to 0]{} 0$.
	For every $\displaystyle \delta' \in \left(0, \frac{1}{4}\right)$, by a straightforward Taylor expansion, we get a constant $C>0$ such that for all $n \ge1$ and for all $\displaystyle u \in \left[ -n^{\delta'}, n^{\delta'} \right]$,
	\begin{equation*}
		\left| \phi^{\ens} \left( \frac{u}{\sqn} \right) - \exp \left( - \frac{s u^2}{2} \right) \right| \le C \frac{u^4}{n} + \frac{u^2}{n}\cdotp
	\end{equation*}
	Take $\displaystyle \delta = \frac{1}{6} > \frac{1}{18}$. Then, for all $n \ge 1$, we have
	\begin{equation}
		\tau(\L,n,\delta) = \int_{-n^\delta}^{n^\delta} \left| \phi^{\lfloor ns \rfloor} \left( \frac{u}{\sqrt{n}} \right) - \exp \left( - \frac{s u^2}{2} \right) \right| \mathrm{d}u \lesssim n^{-\frac{1}{6}}.\label{eq: rade2.}
	\end{equation}
	Hence, by Theorem \ref{thm: theorem principal general.}, and combining \eqref{eq: rade1.} and \eqref{eq: rade2.}, we get that for all $n \ge 1$,
	\begin{equation*}
		\dist_{F.M.}\left(\mu_n(\L),B^{br}\right) \lesssim n^{-\frac{1}{18}} \log(n).
	\end{equation*}
\end{proof}

\subsection{Proof of Theorem \ref{thm: theoreme principal poisson moins 1.}}\label{s: sous-section preuve poisson moins un.}

\begin{proof}[Proof of Theorem \ref{thm: theoreme principal poisson moins 1.}]
	In this proof, we write $a \lesssim b$ if there exists $C>0$ which does not depend on $n$ such that $a \le C b$.
	Assume that $\L$ is the Poisson minus one distribution of parameter $1$. Then, $\L$ is a lattice distribution and, if $X$ is a random variable with distribution $\L$, $P := X + 1$ has a Poisson distribution of parameter $1$. Thus $E[X] = E[P] - 1 = 0$ and $\displaystyle E \left[ X^2 \right] = E \left[ P^2 \right] - 2E[P] + 1 = 1$. As for the proof of Theorem \ref{thm: theoreme principal rademacher.}, one can apply Theorem \ref{thm: theorem principal general.} and, to prove Theorem \ref{thm: theoreme principal poisson moins 1.}, it remains to control $\rho(\L,n)$ and $\tau(\L,n,\delta)$ for some $\displaystyle \delta > \frac{1}{18}$.
	The first remark is that $\L$ has a span $1$ and $\kappa(\L)=1$.
	Then, for every $i \ge 1$, define $P_i := X_i + 1$.
	For all $n \ge 1$,
	\begin{equation*}
		p^n_1(0) = P \left( \sum_{i = 1}^n X_i = 0 \right) = P \left( \sum_{i=1}^n P_i = n \right) = \frac{n^n}{n!} \mathrm{e}^{-n},
	\end{equation*}
	since $\displaystyle \sum_{i=1}^n P_i$ has a Poisson distribution of parameter $n$.
	Now, we use the well-known asymptotic development of the Stirling formula:
	\begin{equation*}
		n! = \sqrt{2 \pi n} \left( \frac{n}{\mathrm{e}} \right)^n \left[1 + \frac{1}{12n} + o \left( \frac{1}{n} \right) \right].
	\end{equation*}
	This gives
	\begin{equation*}
		p^n_1(0) = \frac{1}{\sqrt{2 \pi n}} \left( 1 - \frac{1}{12n} + o \left( \frac{1}{n} \right) \right).
	\end{equation*}
	Recalling that $\displaystyle p_1(0) = \frac{1}{\sqrt{2 \pi}}$, we get
	\begin{equation*}
		\left| \sqrt{n} p^n_1(0) - p_1(0) \right| = \frac{1}{\sqrt{2 \pi}} \left| \frac{1}{12n} + o \left( \frac{1}{n} \right) \right| \lesssim n^{-1},
	\end{equation*}
	and for all $n \ge 1$,
	\begin{equation}
		\tau(\L,n) \lesssim n^{-\frac{17}{18}}.\label{eq: poisson1.}
	\end{equation}
	Furthermore, denote by $\phi$ the characteristic function of $\L$. For every $u \in \R$,
	\begin{equation*}
		\phi(u) = \exp \left( \mathrm{e}^{iu} - iu - 1 \right).
	\end{equation*}
	By a Taylor expansion, for every $n \ge 1$ and $u \in \R$,
	\begin{align*}
		\phi^{\lfloor ns \rfloor} \left( \frac{u}{\sqrt{n}} \right) & %= \exp \left( \lfloor ns \rfloor \left[ \exp \left( i \frac{u}{\sqrt{n}} \right) - 1 - i \frac{u}{\sqrt{n}} \right] \right) \\
		%& = \exp \left( \lfloor ns \rfloor \left[ - \frac{u^2}{2n} - i \frac{u^3}{6n\sqrt{n}} + \frac{u^3}{n \sqrt{n}} \epsilon_1 \left( \frac{u}{\sqrt{n}}\right) \right] \right) \\
		= \exp \left( - \frac{su^2}{2} \right) \exp \left( - i \frac{su^3}{6 \sqrt{n}} + \frac{su^3}{\sqrt{n}} \epsilon_1 \left( \frac{u}{\sqrt{n}}\right) + \{ns\} \frac{u}{\sqrt{n}} \epsilon_2 \left( \frac{u}{\sqrt{n}}\right) \right),
	\end{align*}
	where $\{x\} = x - \lfloor x \rfloor$ for all $x \in \R$, and for all $i \in \{1,2\}$, $\displaystyle |\epsilon_i(x)| \xrightarrow[x \to 0]{} 0$.
	For every $\displaystyle \delta' \in \left(0, \frac{1}{6}\right)$, by the above calculation, we get a constant $C>0$ such that for all $n \ge1$ and for all $\displaystyle u \in \left[ -n^{\delta'}, n^{\delta'} \right]$,
	\begin{equation*}
		\left| \phi^{\ens} \left( \frac{u}{\sqn} \right) - \exp \left( - \frac{s u^2}{2} \right) \right| \lesssim \frac{u^3}{\sqrt{n}}.
	\end{equation*}
	Take $\displaystyle \delta = \frac{1}{12} > \frac{1}{18}$. We get for all $n \ge 1$,
	\begin{equation}
		\tau(\L,n,\delta) = \int_{-n^\delta}^{n^\delta} \left| \phi^{\lfloor ns \rfloor} \left( \frac{u}{\sqrt{n}} \right) - \exp \left( - \frac{s u^2}{2} \right) \right| \mathrm{d}u \lesssim n^{-\frac{1}{6}}.\label{eq: poisson2.}
	\end{equation}
	Hence, by Theorem \ref{thm: theorem principal general.}, and combining \eqref{eq: poisson1.} and \eqref{eq: poisson2.}, we get for all $n \ge 1$,
	\begin{equation*}
		\dist_{F.M.}\left(\mu_n(\L),B^{br}\right) \lesssim n^{-\frac{1}{18}} \log(n).
	\end{equation*}
\end{proof}

\end{document}